\documentclass[12pt]{article}

\usepackage{amsmath, amssymb, amsthm, graphicx, fancyhdr, textcomp, enumerate, diagbox, tcolorbox, esvect, tikz, adjustbox}
\usetikzlibrary{cd}
\usepackage[bottom=0.5in]{geometry}
\graphicspath{ {./images/} }

\usepackage[
color=orange!80, 
bordercolor=black,
textwidth=3cm,
textsize=small,
colorinlistoftodos]
{todonotes}

\newcommand{\Z}{\mathbb{Z}} 

\newcommand{\Q}{\mathbb{Q}}
\newcommand{\N}{\mathbb{N}}

\newcommand{\cI}{\mathcal{I}}

\newcommand{\cL}{\mathcal{L}}

\newcommand{\gen}[1]{\langle#1\rangle} 
\newcommand{\qgen}[1]{\langle\langle#1\rangle\rangle} 
\newcommand{\Aut}{\mathrm{Aut}} 

\newcommand{\normal}{\trianglelefteq} 

\newcommand{\Conj}{\mathrm{Conj}} 

\newcommand{\sq}{\preccurlyeq} 
\newcommand{\psq}{\prec} 
\newcommand{\thru}{\rhd} 
\newcommand{\dthru}{\rhd_{\mathrm{dih}}} 
\newcommand{\bthru}{\inv{\thru}} 
\newcommand{\Inn}{\mathrm{Inn}} 
\newcommand{\Adconj}{\mathrm{Adconj}} 
\newcommand{\Tak}{\mathrm{Tak}}
\newcommand{\AdTak}{\mathrm{AdTak}}

\newcommand{\inv}[1]{#1^{-1}}
\renewcommand{\phi}{\varphi}

\theoremstyle{definition}
\newtheorem{theorem}{Theorem}[section]
\newtheorem{definition}[theorem]{Definition}
\newtheorem{proposition}[theorem]{Proposition}
\newtheorem{claim}[theorem]{Claim}
\newtheorem{lemma}[theorem]{Lemma}
\newtheorem{corollary}[theorem]{Corollary}
\newtheorem{example}[theorem]{Example}

\begin{document}

\medskip                        

\title{Complementation of Subquandles}
\author{K.J. Amsberry \ J.A. Bergquist \ T.A. Horstkamp \ M.H. Lee \ D.N. Yetter}
\date{August 2022}

\thispagestyle{plain}
\maketitle
\begin{abstract}
Saki and Kiani proved that the subrack lattice of a rack $R$ is necessarily complemented if $R$ is finite but not necessarily complemented if $R$ is infinite. In this paper, we investigate further avenues related to the complementation of subquandles.  Saki and Kiani's example of an infinite rack without complements is a quandle, which is neither ind-finite nor profinite. We provide an example of an ind-finite quandle whose subobject lattice is not complemented, and conjecture that profinite quandles have complemented subobject lattices. Additionally, we provide a complete classification of subquandles whose set-theoretic complement is also a subquandle, which we call \textit{strongly complemented}, and provide a partial transitivity criterion for the complementation in chains of strongly complemented subquandles. One technical lemma used in establishing this is of independent interest:  the inner automorphism group of a subquandle is always a subquotient of the inner automorphism group of the ambient quandle.


\end{abstract}

\section{Introduction}
Quandles were first introduced as a complete classical knot invariant up to orientation reversal by Joyce \cite{Joyce1982} in 1982. In recent years, quandles and racks have found numerous applications in knot theory and algebraic topology. Aside from these applications, quandles are also interesting as an algebraic structure. Indeed, quandles are an abstraction of groups, as any group is a quandle under the operation of right conjugation. Quandles can also be constructed from abelian groups using the Takasaki operation \cite{Takasaki1943}.

Like other algebraic structures, quandles have subobjects. These subobjects, or subquandles, are difficult to study because quandles lack constants such as an identity element, so that there is no notion of a kernel for a quandle homomorphism, and thus no analogue of normal subgroups.  Nonetheless, some progress has been made towards understanding interesting properties of subquandles. Recently Saki and Kiani \cite{SakiKiani2021} showed that the lattice of subquandles of a finite quandle is necessarily complemented, but that infinite quandles may have non-complemented subquandles. The present work investigates several problems related to subquandle lattices, subquandle inner automorphism groups, and subquandles whose set-theoretic complements are also subquandles. 


\section{Background and Definitions}

\begin{definition}\cite{Joyce1982}
    A \textit{quandle} is a set $Q$ equipped with binary operations $\thru$ and $\bthru$ satisfying the following axioms for all $x,y,z\in Q$:
    \begin{enumerate}
    \item[Q1.] $x\thru x = x$.
    \item[Q2.] $(x\thru y) \bthru y = x = (x\bthru y) \thru y $.
    \item[Q3.] $(x\thru y)\thru z = (x\thru z)\thru (y\thru z)$.
    \end{enumerate}
\end{definition}
In other words, $\thru$ is idempotent, right-invertible, and right-distributive. Note that the mention of $\bthru$ is not necessary if one is content to describe quandles as models of a first order theory, rather than an equational theory, as Q2 is equivalent to the bijectivity of $(-) \thru y$ for all $y \in Q$. Models of the algebraic structure $(Q,\thru)$ satisfying only Q2 and Q3 are called \textit{racks}. 
\begin{example}
The Tait quandle $(\mathbf{T}_3,\thru)$ is a quandle with underlying set $\{1,2,3\}$ and operation given by:
\begin{figure}[h]
    \centering
    \begin{tabular}{c|c c c} 
        $\thru$ &$1$ &$2$ &$3$ \\
        \hline
         $1$ & $1$ & $3$ & $2$\\
         $2$ & $3$ & $2$ & $1$ \\
         $3$ & $2$ & $1$ & $3$
    \end{tabular}
    \caption{The multiplication table for $\mathbf{T}_3$}
    \label{fig:my_label}
\end{figure}
\end{example}
For more about multiplication tables of quandles, see \cite{HoNelson2005}.
\begin{example}
The transpositions in the symmetric group $S_n$ form a quandle under the operation of conjugation. In fact, the Tait quandle can be realized as the transpositions in $S_3$.
\end{example}

\begin{example}
For any group $G$, defining $\thru$ by $x\thru y = x^{-1}yx$ gives a quandle structure on the same underlying set.  This quandle is denoted $\Conj(G)$.
\end{example}

$\Conj$ is, in fact, a functor from the category of groups to the category of quandles, and admits a left adjoint, denoted $\Adconj$, given by quotienting the free group on the underlying set of the quandle by relations which equate the generator $q\thru r$ with the right conjugate of the generator $q$ by the generator $r$.

The following quandle is of particular relevance to us, as it will appear in several later examples.
\begin{example} 
\label{dihedralKei}
\cite{Takasaki1943}
Suppose $A$ is an abelian group. Then, if we define $x\dthru y = 2y-x$ for any $x,y \in A$, $(A,\dthru)$ forms a quandle. We call this quandle a \textit{Takasaki quandle}, denoted $T(A)$. The operation $\dthru$ of $T(A)$ is called the \textit{dihedral action}.
\end{example}

These last are examples of a special class of quandles studied by Takasaki \cite{Takasaki1943} long before Joyce's work \cite{Joyce1982}:

\begin{definition}
\label{KeiDefinition}
    A {\em kei} or {\em involutory quandle} is a quandle $(K, \thru)$ which satisfies 
    $(x\thru y)\thru y = x$ for all $x,y\in K$.
\end{definition}

\begin{definition}
Given quandles $(Q_1,\thru_1),(Q_2,\thru_2)$, a \textit{quandle homomorphism} is a map $h:Q_1\to Q_2$ satisfying $h(x\thru_1y) = h(x)\thru_2h(y)$ for all $x,y\in Q_1$. 
\end{definition}

A \textit{quandle isomorphism} $\rho$ is a bijective quandle homomorphism, and $\rho$ is a \textit{quandle automorphism} if it maps $Q \to Q$. The set of automorphisms of $Q$ form a group $\Aut(Q)$ under composition. 
\begin{definition}
For any quandle $Q$ and fixed $y\in Q$, the map $S_y : x \to x\thru y$ is called the \textit{symmetry at} $y$. By the second and third axioms, symmetries are quandle automorphisms.  The subgroup of $\Aut(Q)$ generated by the symmetries is denoted $\Inn(Q)$ and is called {\em the inner automorphism group of $Q$}. Its elements --- composites of symmetries and their inverses --- are called {\em inner automorphisms}. 
\end{definition}

By a result of Joyce \cite{Joyce1982}, $\Inn(Q)$ is always a normal subgroup of $\Aut(Q)$ .

\begin{definition}
The \textit{orbit} of $s \in Q$ is $s\cdot \Inn(Q):= \{\rho(s):\rho \in \Inn(Q)\}$. Note $\Inn(Q)$ acts on $Q$ on the right via functional application, i.e., $x\cdot \rho = \rho(x)$. 
\end{definition}

\begin{definition}
A quandle $Q$ is \textit{connected} if the action of $\Inn(Q)$ on $Q$ has only one orbit, namely $Q$ itself. That is, for all $x \in Q$, $x \cdot \Inn(Q) = Q.$ 
\end{definition}

Both $\Aut(Q)$ and $\Inn(Q)$ have an additional relationship with the quandle $Q$ as augmentation groups (as defined by Joyce \cite{Joyce1982}).

\begin{definition}
An {\em augmentation of a quandle} $Q$ is a map $|\;|:Q\rightarrow G$ from $Q$ to a group $G$, together with a right action of $G$ on $Q$ by quandle homomorphisms, satisfying 

\begin{itemize} 
\item $q\thru r = q|r|$, and 
\item $|\;|$ is $G$-equivariant with respect to this action and the action of $G$ on itself by right conjugation.
\end{itemize}
\end{definition}

Joyce \cite{Joyce1982} actually defines an ``augmented quandle'' as a $G$-set with the equivariance condition and the condition that $p|p| = p$, then induces a quandle structure on the set.  For our purposes, where the quandle is already given, the definition above is equivalent to the instance of Joyce's notion which induce the given quandle.

The map $q\rightarrow S_q$ gives an augmentation of any quandle $Q$ in $\Inn(Q)$ and in $\Aut(Q)$. Also, any quandle $Q$ admits an augmentation in $\Adconj(Q)$, with the map $|\;|:Q\rightarrow \Adconj(Q)$ given by inclusion of generators.  As shown in \cite{Joyce1982}, this is in fact the initial augmentation of $Q$, so that the fifth author has always favored calling $\Adconj(Q)$, ``the universal augmentation group of $Q$.''

\begin{definition}
A subset $Q'$ of a quandle $(Q,\thru)$ is a \textit{subquandle} of $Q$ if it is also a quandle under the restriction of the same operation. In particular, to show $Q'$ is a subquandle of $Q$, it suffices to show $Q'$ is closed under $\thru|_{Q'}.$ We write $Q'\sq Q$ and $Q' \psq Q$ to denote that $Q'$ is a subquandle and a proper subquandle of $Q$, respectively.
\end{definition}

It is clear that orbits under the action of the inner automorphism group are subquandles. 
In fact, such orbits are of prime importance to quandle theory, as Ehrman et. al. \cite{EGTY2008} proved that they encode the structure of the ambient quandle. We shall briefly reiterate their definitions for a more explicit description of this result.
\\\\
Let $(Q,\thru)$ be any quandle, and let $\{Q_1,\dots,Q_n\}$ be the set of orbits of the action of $\Inn(Q)$ on $Q$. Suppose $Q$ is not connected, as otherwise the following idea is not interesting. Then, Theorem 3.1 of Ehrman, et al. \cite{EGTY2008} states that $Q$ may be expressed \textit{uniquely} as $\#(Q_1, \dots, Q_n, M)$ for some $n\times n$ matrix of quandle homomorphisms $M$, where $\thru_i=\thru|_{Q_i\times Q_i}$, and the ``semi-disjoint union" operation $\#$ is defined in Definition \ref{semi-disjoint} below.

 For each quandle $(Q_i,\thru_i)$, let the universal augmentation map $|\;|_{Q_i}:Q_i \to \Adconj(Q_i)$ be such that $x|y|_{Q_i}:=x\thru_i y$ for any $x,y\in Q_i.$ Then, for each $1\leq i \leq n$, by the universal property of $\Adconj(Q_i)$, there is a unique group homomorphism $\phi_i:\Adconj(Q_i)\to \Aut(Q_i)$ such that the following diagram commutes (note $S:Q_i\to\Inn(Q_i)$ via $S(q_i)=S_{q_i}$ is the symmetry mapping):

\begin{figure}[h]
    \centering
    \adjustbox{center}{
    \begin{tikzcd}[row sep=large,column sep=huge]
    Q_i \arrow[rd,"S"] \arrow[r, "|\;|_{Q_i}"] & \Adconj(Q_i) \arrow[d,dotted,"\phi_i"] \\
    & \Aut(Q_i)
    \end{tikzcd}}
    \caption{Canonical group homomorphism $\phi_i$}
    \label{fig:my_label}
\end{figure}
Call $\phi_i$ the \textit{canonical group homomorphism} from $\Adconj(Q_i)$ to $\Aut(Q_i).$

\begin{definition} \label{semi-disjoint} \cite{EGTY2008} For each $1\leq i,j\leq n$, let $g_{i,j}:\Adconj(Q_i)\to\Aut(Q_j)$ be a group homomorphism such that $g_{i,i}$ is the canonical group homomorphism $\phi_i$. Also let $M=(g_{i,j})$ be the $n\times n$ matrix of group homomorphisms with entries $g_{i,j}$. Then, 
$$\#(Q_1, \dots, Q_n, M):= \left( \coprod_{i=1}^n Q_i, \thru \right),$$
where $x\thru y:= xg_{i,j}(|y|)$ for $x\in Q_j$ and $y \in Q_i.$
\end{definition}
For arbitrary quandles $Q_1,\dots, Q_n$, if $M$ satisfies compatibility conditions given in \cite{EGTY2008}, then $M$ is called a {\em mesh}, and $\#(Q_1, \dots, Q_n, M)$ with $\thru$ as just defined is a quandle, called the {\em semi-disjoint union of $Q_1,\ldots Q_n$ with respect to $M$}.

\begin{theorem}\cite{EGTY2008}\label{meshcriterion}
    A matrix of group homomorphisms $M = (g)_{ij}:\Adconj(Q_i)\to \Aut(Q_j)$ is a \textit{mesh} if and only if the following conditions are satisfied are met for all $i,j,k$ distinct with $ 1\le i,j,k\le n$,
    \begin{enumerate}
        \item $(xg_{ji})(|y|_{Q_j})\thru_{i}z = (x\thru_{i}z)g_{ji}(|yg_{ij}(|z|_{Q_i})|)_{Q_j}$,
        \item $ 
        (xg_{ji}(|y|_{Q_j}))g_{ki}(|z|_{Q_k}) 
        = 
        (xg_{k,i}(|z|_{Q_k}))(g_{ji}(|yg_{k,j}(|z|_{Q_k})|_{Q_j})).
        $
    \end{enumerate}
\end{theorem}

More importantly, the same authors proved that there is always a suitable mesh to decompose a quandle into the semisdisjoint union of its orbits under the inner automorphism group,

\begin{theorem}
    \cite{EGTY2008}
    \label{orbitDecomp}
    For a quandle $Q$ with $n$ orbit decompositions under $\Inn(Q)$, labeled $Q_1,\dots, Q_n$, there exists a mesh $M$ such that 
    \[
    Q \cong \#(Q_1,\dots, Q_n, M).
    \]
\end{theorem}

\section{Subquandle Lattices}

For a quandle $Q$ and $A\subseteq Q$, we write $\qgen{A}$ to denote the subquandle of $Q$ generated by $A$. As shown by Saki and Kiani in \cite{SakiKiani2021}, the set of subquandles of any quandle $Q$ under inclusion forms a lattice, which we denote by $\cL(Q)$. In particular, the two lattice operations \textit{meet} and \textit{join}, which will be described momentarily, are well-defined. Given two subquandles $Q_1, Q_2\sq Q$, their meet is $Q_1\wedge Q_2 = Q_1\cap Q_2$ and their join is $Q_1\vee Q_2 =  \qgen{Q_1\cup Q_2}$. Note that $Q_1\wedge Q_2$ is the largest subquandle of in $Q$ contained in $Q_1$ and $Q_2$, while $Q_1\vee Q_2$ is the smallest subquandle containing $Q_1$ and $Q_2$. 
\\\\
It is also worth noting that $Q_1\vee Q_2$ is not necessarily $Q_1\cup Q_2$, as $Q_1 \cup Q_2$ may not be a subquandle of $Q$. For example, if $Q$ is the Tait quandle on $\{1,2,3\}$, $\{1\}\sq Q$ and $\{2\}\sq Q$ but $\{1\}\vee \{2\} = Q$.
\begin{definition}
The subquandle lattice $\cL(Q)$ of a quandle $Q$ is \textit{complemented} if for any $Q_1\sq Q$, there exists some $Q_2\sq Q$ such that $Q_1\vee Q_2 = Q$ and $Q_1 \wedge Q_2 = \emptyset$. 
\end{definition}

In \cite{SakiKiani2021}, Saki and Kiani proved the following results.

\begin{theorem}\cite{SakiKiani2021}
\label{sakikianicomplemented}
The subrack lattice of a finite rack is complemented.
\end{theorem}

\subsection{Strongly Complemented Subquandles}

In this section, we investigate a special class of complemented subquandles: subquandles whose set-theoretic complement is also a subquandle. As we will show, these subracks are particularly interesting in light of \cite{EGTY2008}.

\begin{definition}
Let $Q$ be a quandle and $Q'\sq Q$. Then $Q'$ is \textit{strongly complemented} in $Q$ if $Q\setminus Q' \sq Q$. 
\end{definition}

To motivate this definition, let us consider three examples. 

\begin{example}~
\begin{enumerate}
    \item Consider the integers $\Z$ under the dihedral operation $\dthru$ as described in Example \ref{dihedralKei}. Here, the even integers form a subquandle, as do the odd integers. Thus, both the even integers and the odd integers are strongly complemented subquandles of $\Z$.
    \item Let $G$ be a group, and let $\{C_\alpha\}_{\alpha\in\Lambda}$ be any collection of conjugacy classes of $G$. Because the conjugacy class of $G$ form a partition of $G$, $\bigcup_{\alpha\in\Lambda} C_\alpha \sq \Conj (G)$ is a strongly complemented subquandle of $G$.
    \item For a quandle $Q$, any union of orbits under the action of the inner automorphism group $\Inn(Q)$ is strongly complemented in $Q$. 
\end{enumerate} 

\end{example}


We will prove that strongly complemented subquandles of a quandle $Q$ can also be described in terms of actions of $\Inn(Q)$ --- this description will be made explicit in Theorem \ref{equivalenceTheorem}. Before proceeding, we describe the two different actions of $\Inn(Q)$ which will be used throughout this paper. We shall make extensive use of both of these actions, often together, so it is essential to have an efficient way of distinguishing between them.
\\\\
The first action is the usual right action $\cdot: Q\times \Inn(Q)\to Q$ given by $q\cdot \sigma = \sigma(q).$ As described previously, the action of $\Inn(Q)$ creates an orbit for each $q\in Q$, which we denote by $q\cdot \Inn(Q)$. In a similar way, for $Q'\sq Q$, we use the notation $Q'\cdot \Inn(Q)$ to refer to the union of the orbits of elements of $Q'$, i.e., $Q'\cdot\Inn(Q) = \bigcup_{q\in Q'} q\cdot \Inn(Q).$ We refer to this as the \textit{orbit closure} of $Q'$. There is more to this name than simply the fact that the orbit closure is closed under the action of $\Inn(Q)$. Recall:

\begin{definition}
    A {\em closure operation} on a partially ordered set $(X,\leq)$ is a function $\mathrm{cl}:X\rightarrow X$ which is inflationary, that is $x \leq \mathrm{cl}(x)$, and idempotent, that is $\mathrm{cl}(\mathrm{cl}(x)) = \mathrm{cl}(x)$.
\end{definition}

We then have:

\begin{proposition}  Orbit closure is a closure operation on the lattice of subquandles.
\end{proposition}

For the second action, suppose $Q' \sq Q$. The second action is the right action $\cdot: \cL(Q)\times \Inn(Q)\to \cL(Q)$ given by $Q'\cdot \sigma = \sigma[Q']$, or the image of $Q'$ under $\sigma$. Thus, we use the notation $[Q']\cdot \Inn(Q)$ to refer to the set of images of $Q'$ under elements of $\Inn(Q)$, i.e., $[Q']\cdot \Inn(Q)=\{Q'\cdot \sigma :q\in Q'\}.$ We use this notation to capture the action on $\Inn(Q)$ on the subquandle lattice $\cL(Q).$
\\\\
Finally, for ease of notation, we will write $q\sigma$ instead of $q\cdot\sigma$ and $Q'\sigma$ instead of $Q'\cdot \sigma$, where $q \in Q$ and $Q'\sq Q$. With our two actions defined, we state our classification of strongly complemented subquandles:

\begin{theorem}
\label{equivalenceTheorem}
Let $Q$ be a quandle, and let $Q'\sq Q$. Denote the subquandle lattice of $Q$ by $\mathcal{L}(Q)$. The following are equivalent:
\begin{enumerate}
    \item $Q\setminus Q'\sq Q$,
    \item $Q'$ is a union of orbits under the action of $\Inn(Q)$ on $Q$,
    \item $Q'$ is a fixed point of the action of $\Inn(Q)$ on $\mathcal{L}(Q)$,
    \item $Q = \#(Q',Q\setminus Q', M)$ for a mesh $M$.
\end{enumerate}
\end{theorem}
\begin{proof} 
    \item[(2) $\Leftrightarrow$ (3):] First, suppose that $Q'=Q\cdot \Inn(Q)$, so $Q' = \bigcup_{i\in \cI} Q_i$ for some indexing set $\cI$ and each $Q_i$ an orbit of the action of $\Inn(Q)$ on $Q$. Now, pick any $x\in Q'$ and any $\sigma\in \Inn(Q)$. Then, $x\in Q_i$ for some $i \in \cI$, and $x\sigma \in Q_i\subseteq Q'$ because $Q_i$ is an orbit. 
    Since $x\in Q'$ was arbitrary, $Q' \sigma \subseteq Q'$. Also, since $\sigma$ was an arbitrary inner automorphism, $Q'\sigma^{-1}\subseteq Q'$, which implies $x\sigma^{-1}\in Q'$. Also, $x=(x\sigma^{-1})\sigma$, so $Q'\subseteq Q'\sigma$. Thus, since $\sigma \in \Inn(Q)$ was arbitrary, $[Q']\cdot\Inn(Q) = \{Q'\sigma:q \in Q'\} = \{Q'\}$, as desired.
    \\\\
    Now suppose that $[Q']\cdot\Inn(Q) = \{Q'\}$, and pick any $x\in Q'$. Because $x\sigma \in Q'$ for any $\sigma \in \Inn(Q)$, $x\cdot \Inn(Q)=\{x \sigma:\sigma \in \Inn(Q)\}\subseteq Q'$. Since $x\in Q'$ was arbitrary, $Q'\cdot \Inn(Q) \subseteq Q'.$
    Moreover, $Q'\subseteq Q\cdot \Inn(Q),$ as each $x =x\cdot \mathrm{id}_{\Inn(Q)}\in x \cdot \Inn(Q)$. Thus, $Q'=Q\cdot\Inn(Q)$, as desired.
    
    Before proceeding, we establish a useful lemma:
    \begin{lemma}
    If $Q'$ is not a fixed point under the action of $\Inn(Q)$ on $\cL(Q)$, then there exists some $x\in Q\setminus Q'$ such that $Q'\cdot S_x\neq Q'$. In other words, there exists some $x \in Q\setminus Q'$ and some $y\in Q'$ such that $y\thru x\not\in Q'$. 
    \end{lemma}

\begin{proof}
Since $[Q']\cdot\Inn(Q)\neq\{Q'\}$, by definition, there exists some $\sigma\in \Inn(Q)$ such that $Q'\cdot \sigma \ne Q'$. Now, assume for the sake of contradiction that all of the symmetries of elements in $Q$ fix $Q'$. Then the composition of any number of symmetries and symmetry inverses must also fix $Q'$. Hence, all elements of $\Inn(Q)$ must fix $Q'$, a contradiction. So, there must exist some $x$ whose symmetry does not fix $Q'$. Accordingly, there exists some $y\in Q'$ such that $y\thru x\not\in Q'$. Clearly, $x$ cannot be in $Q'$ as $Q'$ is a subquandle and is closed under $\thru$. Thus, $x\in Q\setminus Q'$, as desired.
\end{proof}
    (2) $\Leftrightarrow$ (1): If $Q'$ is a union of orbits, so is $Q\setminus Q'$. Since unions of orbits are subquandles, $Q'$ is strongly complemented. For the converse, assume for the sake of contradiction that $Q'\sq Q$ is strongly complemented but not a union of orbits. Because $Q'$ is not a union of orbits, $Q'$ is not a fixed point under the action of $\Inn(Q)$ on $\cL(Q)$. Hence, by the previous lemma, there exists some $q\in Q'$ and some $y\in Q\setminus Q'$ such that  $q\thru y\in Q\setminus Q'$. Because $Q\setminus Q'$ is a subquandle, it is closed under $\bthru$, so $q=(q\thru y) \bthru y\in Q\setminus Q'$. Hence, $q \in Q'\cap (Q\setminus Q'),$ a contradiction. 
    \\\\
    (1) $\Leftrightarrow$ (4): If $Q'$ is strongly complemented, $Q'$ is a union of orbits. Because Theorem \ref{orbitDecomp} also holds for semidisjoint unions of unions of orbits, we obtain $Q = \#(Q',Q\setminus Q', M)$ for some mesh $M$. The converse follows directly from the definition of the semidisjoint union.
\end{proof}
As an immediate corollary, we have that the only strongly complemented subquandles of a connected quandle $Q$ are $Q$ and the empty quandle, as $Q$ is only orbit of the action of $\Inn(Q)$ on $Q$. Another easy corollary is that the the sublattice consisting of strongly complemented subquandles is boolean algebra.

\subsection{Chains of Strongly Complemented Subquandles}

We now consider chains of strongly complemented quandles. Theorem 4.9 below shows that any subquandle nested between a strongly complemented subquandle and its ambient quandle is also strongly complemented. We also also provide a transitivity criterion for strong complementedness in Theorem 4.13. First, we establish a useful lemma of independent interest:
\begin{lemma}
\label{Subquotient}
Let $Q$ be a quandle and suppose $Q'\sq Q$. Then, there exists some subgroup $S_{Q'}\le \Inn(Q)$ and some normal subgroup $K_{Q'}\normal S_{Q'}$ such that $\Inn(Q')\cong S_{Q'}/K_{Q'}$.
\end{lemma}
In other words, the inner automorphism group of a subquandle is a subquotient of the inner automorphism group of the ambient quandle. 
\begin{proof}
Define $S_{Q'}:=\gen{\{S_{q}: q\in Q'\}}\leq \Inn(Q)$, where each $S_{q}\in \Inn(Q).$ Consider the restriction map $\tau:S_{Q'}\to \Inn(Q')$ defined by $\tau(f)=f|_{Q'}.$
\begin{claim} $\tau$ is well-defined. \end{claim}
\begin{proof}
For $z\in Q$ and $\epsilon\in \{-1,1\}$, we write $S_z^{\epsilon}$ to mean the symmetry $(x\mapsto x\thru^{\epsilon}z) \in \Inn(Q).$ We need two fairly obvious properties of bijections which preserve a given subset:
First, let $n\in\N$, and let $Q'\subseteq Q$ be sets. For each $1\leq i\leq n$, let $f_i : Q\to Q$ satisfy $f_i[Q']\subseteq Q'.$ Then,     $$(f_n\circ f_{n-1}\circ\cdots\circ f_1) |_{Q'}=(f_n|_{Q'})\circ (f_{n-1}|_{Q'})\cdots\circ (f_1|_{Q'}).$$
    Note that each restriction $f_i|_{Q'}$ is well-defined by the assumption $f_i[Q']\subseteq Q'.$
Second, let $Q'\subseteq Q$ be sets, and suppose $f: Q \to Q$ is a bijection such that $f[Q']= Q'$. Then $(f|_{Q'})^{-1}=\inv{f}|_{Q'}.$\\\\
The first fact allows us to freely move restrictions to $Q'$ in and out of compositions of symmetries of $Q$. The second fact allows us to freely move inverses of symmetries of $Q$ in and out of restrictions to $Q'$. Now, to show $\tau$ maps every element of $S_{Q'}$ into $\Inn(Q')$, suppose $f\in S_{Q'}$ has a decomposition $f=S_{x_n}^{\epsilon_n}\circ S_{x_{n-1}}^{\epsilon_{n-1}} \cdots \circ S_{x_1}^{\epsilon_1}$
for some $n\in \N$ and each $x_i\in Q'$, $\epsilon_i\in\{-1,1\}$. Then, 
\begin{align*}
    \tau(f) &=  (S_{x_n}^{\epsilon_n}\circ S_{x_{n-1}}^{\epsilon_{n-1}}\circ \cdots \circ S_{x_1}^{\epsilon_1})|_{Q'} \\
    &= (S_{x_n}^{\epsilon_n} |_{Q'})\circ \cdots \circ(S_{x_1}^{\epsilon_1} |_{Q'}) \\
    &= (S_{x_n} |_{Q'})^{\epsilon_n}\circ \cdots \circ(S_{x_1} |_{Q'})^{\epsilon_1} \in \Inn(Q').
\end{align*}
Hence, $\tau[S_{Q'}]\subseteq \Inn(Q')$. 
\end{proof}
\begin{claim}
$\tau$ is surjective.
\end{claim}
\begin{proof}
Fix $\alpha\in\Inn(Q')$, and suppose $\alpha=T_{x_n}^{\epsilon_n}\circ\cdots\circ T_{x_1}^{\epsilon_1}$, where each $T_{x_i}^{\epsilon_i}\in \Inn(Q')$ is defined via $T_{x_i}^{\epsilon_i}(x):=x \thru^{\epsilon_i} x_i$. Note that $S_{x_n}^{\epsilon_n}\circ \cdots \circ S_{x_1}^{\epsilon_1}$ is an element of $\Inn(Q).$ Moreover,
\begin{align*}
    \tau(S_{x_n}^{\epsilon_n}\circ \cdots \circ S_{x_1}^{\epsilon_1}) &= (S_{x_n}^{\epsilon_n}\circ\cdots\circ S_{x_1}^{\epsilon_1})|_{Q'} \\
    &= S_{x_n}^{\epsilon_n}|_{Q'}\circ\cdots\circ S_{x_1}^{\epsilon_1}|_{Q'} \\ &=(S_{x_n}|_{Q'})^{\epsilon_n}\circ\cdots\circ(S_{x_1}|_{Q'})^{\epsilon_1} \\
    &= T_{x_n}^{\epsilon_n}\circ\cdots\circ T_{x_1}^{\epsilon_1} = \alpha.
\end{align*}
\end{proof}
\begin{claim}
$\tau$ is a group homomorphism.
\end{claim}
\begin{proof}
Let $f,g\in S_{Q'}$ be arbitrary, and suppose $f=S_{x_n}^{e_n}\circ \cdots \circ S_{x_1}^{e_1}$ and $g=S_{y_m}^{f_m}\circ \cdots \circ S_{y_1}^{f_1},$ such that $m,n\in \N$, and each $x_i,y_j\in Q'$, $e_i,f_j\in\{-1,1\}$. Since each symmetry fixes $Q'$, $f$ and $g$ must also fix $Q'$, as they are compositions of symmetries and inverse symmetries. Hence, by Fact 1, $\tau(f\circ g)=(f\circ g)|_{Q'}=f|_{Q'}\circ g|_{Q'}=\tau(f)\circ\tau(g)$. 
\end{proof}
Therefore, $\tau$ is a surjective group homomorphism. By the First Isomorphism Theorem, 
$S_{Q'}/\ker(\tau) \cong \Inn(Q'),$ as desired.
\end{proof}
\begin{theorem}
Suppose $Q''\sq Q'\sq Q$ and that $Q''$ is strongly complemented within $Q$. Then, $Q''$ is strongly complemented in $Q'$.
\end{theorem}

\begin{proof}
By Theorem \ref{equivalenceTheorem}, it suffices to show that $Q''$ is a fixed point under the action of $\Inn(Q')$ on $\cL(Q')$. To that end, fix arbitrary $\sigma \in \Inn(Q')$. Consider $S_{Q'}:=\langle S(Q')\rangle \le \Inn(Q)$, where $S(Q')$ denotes the image of $Q'$ under the function which maps elements to their symmetries. Since the restriction map $\tau:S_{Q'}\to \Inn(Q')$ from Lemma \ref{Subquotient} is surjective, there exists some $\sigma' \in S_{Q'}$ extending $\sigma.$ Then, because  $Q''$ is strongly complemented in $Q$, by Theorem \ref{equivalenceTheorem}, it is a fixed point of the action of $\Inn(Q)$ on $\cL(Q).$ In particular, $Q''\sigma' = Q''$. But, since $\sigma'$ extends $\sigma:Q'\to Q'$ and since $Q''\subseteq Q'$, we must also have $Q''\sigma = Q''$. Thus, since $\sigma \in \Inn(Q)$ was arbitrary, $Q''$ is indeed a fixed point under the action of $\Inn(Q')$ on $\cL(Q')$, as desired.
\end{proof}
To establish our transitivity criterion, we will need several technical lemmas.
\begin{lemma}
\label{orbitClosureSubsetLemma}
Suppose $Q''\sq Q'\sq Q$ such that $Q''$ is strongly complemented within $Q'$ and $Q'$ is strongly complemented within $Q$. Then, $Q''\cdot \Inn(Q)\subseteq Q'$.
\end{lemma}
\begin{proof}
By the definition of $Q''\cdot \Inn(Q),$ it suffices to show that $q\cdot\Inn(Q)\subseteq Q'$ for any $q\in Q''.$ To that end, fix arbitrary $q \in Q''$ and $\sigma \in \Inn(Q)$. Since $Q'$ is strongly complemented within $Q$, by Theorem \ref{equivalenceTheorem}, $[Q']\cdot \Inn(Q)=\{Q'\},$ so  $Q'f = Q'$ for any $f\in \Inn(Q)$. In particular, since $q \in Q'$ and $\sigma\in \Inn(Q)$, $q\sigma \in Q'$. But since $\sigma\in \Inn(Q)$ was arbitrary, it follows that $q \cdot \Inn(Q)\subseteq Q'$, as desired.
\end{proof}
It turns out that $Q\setminus (Q''\cdot \Inn(Q))$ will complement $Q''$ in $\cL(Q)$, which the above lemma will help show. 
\begin{lemma}
Suppose $Q''\sq Q' \sq Q$ such that $Q''$ is strongly complemented within $Q'$ and $Q'$ is strongly complemented within $Q$. Then, for any $\gamma\in \Inn(Q)$, $Q''\gamma$ is strongly complemented within $Q'$.
\end{lemma}

\begin{proof}
 Since $Q''$ is strongly complemented in $Q'$ and $Q'$ is strongly complemented in $Q$, by Theorem \ref{equivalenceTheorem}, we have $Q'\setminus Q''\sq Q'$ and $[Q']\cdot \Inn(Q) =\{Q'\}$. Fix arbitrary $\gamma \in \Inn(Q)$. Since the images of subquandles of $Q$ under inner automorphisms of $Q$ are still subquandles of $Q$, $(Q'\setminus Q'')\gamma\sq Q$. Also, because $Q'\setminus Q''\subset Q'$ and $\gamma:Q'\to Q',$ $(Q'\setminus Q'') \gamma\subset Q'$. But, since we showed $ (Q'\setminus Q'')\gamma$ is a quandle, $ (Q'\setminus Q'')\gamma\sq Q'$. We also know $Q' \gamma = Q'$ because $[Q']\cdot \Inn(Q) =\{Q'\}$. Also, since $\gamma$ is a bijection, $(Q'\setminus Q'') \gamma = Q' \gamma\setminus Q''\gamma$. Combining the previous two statements, we have $(Q'\setminus Q'')\gamma = Q'\setminus Q''\gamma \sq Q'$. Thus, since its set-theoretic complement is a subquandle of $Q'$, $Q''\gamma$ is strongly complemented within $Q'$.
\end{proof}

\begin{lemma}
\label{removalLemma}
Suppose $Q''\sq Q' \sq Q$ such that $Q''$ is strongly complemented within $Q'$ and $Q'$ is strongly complemented within $Q$. Let $\sigma\in \Inn(Q)$. Then there exists some $\sigma'\in \gen{\{S_x : x \in Q\setminus Q'\}}\leq \Inn(Q)$ such that $Q''\sigma = Q''\sigma'$.
\end{lemma}

\begin{proof}
Since $\sigma\in \Inn(Q)$, we may fix a representation $\sigma = S_{x_1}S_{x_2}\cdots S_{x_n}$ of $\sigma$ as a string of symmetries, where each $x_i\in Q$ and $\epsilon_i \in \{-1,1\}$. Our strategy will be to show that by removing all instances of symmetries at elements in $Q'$ which appear in $S_{x_1}S_{x_2}\cdots  S_{x_n}$, we obtain another element $\sigma' \in \Inn(Q)$ which acts the same as $\sigma$ upon $Q''$ in $\cL(Q)$.
\\\\
 A systematic way to construct $\sigma'$ from the string $S_{x_1}S_{x_2}\cdots S_{x_n}$ is by checking if $x_1$ is in $Q'$ and removing $S_{x_1}$ iff it is, then similarly checking $x_2$ and possibly removing $S_{x_2}$, and so on. Assume for the sake of contradiction that in doing this process, there is an index $k$ such that $x_k\in Q'$ and that when $S_{x_k}$ is removed from the string $S_{x_1}S_{x_2}\cdots S_{x_n}$, the resulting inner automorphism acts differently than $\sigma$ does upon $Q''$. Let $i$ be the first such index $k$, so for all $1\leq j<i$ such that $x_j\in Q'$, removing $S_{x_j}$ doesn't change the resulting action on $Q''$, but removing $S_{x_i}$ does.
\\\\
Then, let $\gamma = S_{x_1}\cdots S_{x_{i-1}}$ and $\gamma' = S_{x_{i +1}}\cdots S_{x_n}$, so we have 
$\sigma = \gamma S_{x_i}\gamma'$. By the associativity of group action, we have
$ Q''\sigma = ((Q''\gamma)S_{x_i})\gamma'$. Moreover, by the previous lemma, $Q''\gamma$ is strongly complemented within $Q'$. Hence, $Q''\gamma$ is a fixed point of the action of $\Inn(Q')$ on $\mathcal{L}(Q')$, as follows from Theorem \ref{equivalenceTheorem}. Additionally, since $x_i\in Q'$, we know that $S_{x_i}|_{Q'}$ is an element of $\Inn(Q')$, and moreover, we have $xS_{x_i}|_{Q'}=xS_{x_i}$ for any $x \in Q'$. From this and the previous statement, it follows that $(Q''\gamma)S_{x_i} = (Q''\gamma)S_{x_i}|_{Q'} = Q''\gamma$. Now, substituting, we have the identity \[
Q''\sigma = ((Q''\gamma)S_{x_i})\gamma' = (Q''\gamma)\gamma' = Q''(\gamma\gamma').
\]
Hence, removing $S_{x_i}$ doesn't change the action of the resulting string (namely $\gamma\gamma'$) on $Q''$, contradicting our assumption that it did. So, no such index $i$ can exist, and hence removing all symmetries at elements in $Q'$ from $S_{x_1}S_{x_2}\cdots  S_{x_n}$ also doesn't resulting string's, namely $\sigma'$'s, action on $Q''$. 
\end{proof}

\begin{theorem}
Let $Q''\sq Q' \sq Q$ such that $Q''$ is strongly complemented within $Q'$, while $Q'$ is strongly complemented within $Q$. Then $Q''$ is complemented within $Q$. 
\end{theorem}

\begin{proof}
Since $Q''\cdot \Inn(Q)$ is a union of orbits by construction, by Theorem \ref{equivalenceTheorem}, it follows that $Q\setminus (Q''\cdot \Inn(Q))\sq Q$. And since $Q''\subseteq Q''\cdot\Inn(Q)$, $Q''\cap Q\setminus (Q''\cdot \Inn(Q))=\emptyset$. Thus, to show $Q\setminus (Q''\cdot \Inn(Q))$ is a complement for $Q''$ in $\cL(Q)$, it remains to show that $Q''\vee Q\setminus(Q''\cdot \Inn(Q)) = Q$. To do this, it suffices to show that $Q''\cdot \Inn(Q)\subseteq Q''\vee Q\setminus(Q''\cdot \Inn(Q))$ --- if this holds, we will have $$\underbrace{(Q\setminus (Q''\cdot\Inn(Q)))\cup (Q''\cdot \Inn(Q))}_Q\subseteq Q''\vee Q\setminus(Q''\cdot \Inn(Q)),$$ 
which implies $Q= Q''\vee Q\setminus(Q''\cdot \Inn(Q)).$
\\\\
To that end, choose any $x\in Q''\cdot \Inn(Q)$. Then, there is some $\sigma \in \Inn(Q)$ such that $x\in Q''\sigma$. By Lemma \ref{removalLemma}, there exists some $\sigma'\in \gen{\{S_x : x \in Q\setminus Q'\}}$ such that $Q''\sigma' = Q''\sigma$. Fix $q_1,\dots, q_n\in Q\setminus Q'$ and $\epsilon_1,\dots, \epsilon_n \in \{-1, 1\}$ such that $\sigma' = S_{q_1}^{\epsilon_1}\dots S_{q_n}^{\epsilon_n}$. By the definition of the action of $\Inn(Q)$, $Q''\sigma' = \{q\thru^{\epsilon_1}q_1\dots \thru^{\epsilon_n}q_n: q\in Q''\}$. Thus, $x = q \thru^{\epsilon_1}q_1\dots \thru^{\epsilon_n}q_n$ for some $q \in Q''$. In particular, $x \in \qgen{Q''\cup (Q\setminus Q')}$. Additionally, we know $Q''\cdot\Inn(Q)\subseteq Q'$ by Lemma \ref{orbitClosureSubsetLemma}, so $Q\setminus Q'\subseteq Q\setminus (Q''\cdot \Inn(Q))$. Therefore,
\begin{align*}
    x \in \qgen{Q''\cup (Q\setminus Q')} 
    &\subseteq  \qgen{Q''\cup (Q\setminus (Q''\cdot\Inn(Q)))} = Q''\vee Q\setminus (Q''\cdot\Inn(Q))
\end{align*}
by the definition of join. But, since $x\in Q''\cdot\Inn(Q)$ was arbitrary, it follows that $Q''\cdot \Inn(Q) \subseteq Q''\vee Q\setminus (Q''\cdot \Inn(Q))$, as desired.
\end{proof}

This theorem not only shows that $Q''$ is complemented when it lies two links down in a finite chain of strongly complemented subquandles of $Q$ --- it also provides an explicit complement for $Q''$, namely $Q\setminus (Q\cdot \Inn(Q))$. Moreover, since this complement of $Q''$ is a union of orbits, it is strongly complemented within $Q$, namely by $Q\cdot\Inn(Q)$.


An natural question to consider is  whether complementedness is carried down further links in a finite chain of strongly complemented subquandles. More explicitly, suppose we have a chain of strongly complemented subquandles, $Q_3\sq Q_2\sq Q_1\sq Q$. A natural attempt would be to consider $Q\setminus (Q_3\cdot \Inn(Q))$ as a possible complement of $Q_3$ in $\cL(Q)$. However, the join requirement for this to be a complement is not nearly as straightforward to verify. It is our hope that future research will prove that complementedness is transitive across chains of strongly complemented of arbitrary finite length, or else find a counterexample.

\subsection{Infinite Quandles}

The question of whether the subquandle lattice of a finite quandle is complemented has been settled by previous research. More explicitly, since quandles are racks, the following theorem shows that the subquandle lattices of finite quandles are necessarily complemented. 
\begin{theorem}\cite{SakiKiani2021}
\label{finiteComplemented}
Let $R$ be a finite rack. Then the following hold:
\begin{enumerate}
    \item The intersection of all the maximal subracks of $R$ is empty.
    \item Suppose that $R_1, R_2\sq R$. Then $R_1$ is complemented in $\mathcal{R}(R_1\vee R_2)$ by some subset of $R_2$.
\end{enumerate}
\end{theorem}
\begin{corollary} \cite{SakiKiani2021}
If $R$ is a finite rack, then $\mathcal{R}(R)$ is complemented. 
\end{corollary}

The natural follow-up to this result is whether infinite racks, and hence infinite quandles, have necessarily complemented subobject lattices. Some infinite racks, such as trivial racks, possess complemented subobject lattices. However, the following theorem demonstrates that this is not always the case:

\begin{theorem}\cite{SakiKiani2021} \label{QRacknotComp}
The rack $(\Q,\dthru)$ does not have a complemented subrack lattice. In particular, the singleton $\{0\}\sq \Q$ lacks a complement.
\end{theorem}

To understand the proof of Theorem \ref{QRacknotComp}, we first need the following definition.
\begin{definition}
 We say $M\psq Q$ is \textit{maximal} if there is no $N \psq Q$ such that $M \psq N$. Equivalently, $Q$ is the only proper extension of $M$ in $\cL(Q).$ 
\end{definition}
Observe that since singletons are always subquandles, any maximal subquandle $M$ of a quandle $Q$ must be complemented by $\{x\}$ for any $x\in Q\setminus M$. However, this is not true for racks, as racks may not necessarily have singleton subracks. 
\\\\
The proof of Theorem \ref{QRacknotComp} relies the containment of all subracks of a finite rack $R$ in some maximal subrack of $R$ \cite{SakiKiani2021}. In particular, the ``top layer" of the subrack lattice of $R$ is composed of a set of pairwise-disjoint maximal subracks under which every other proper subrack of $R$ lies. However, this strategy is no longer useful for infinite quandles, as the following result shows that maximal subquandles may not even exist.
\begin{proposition}
\label{QNoMaximal}
$(\Q, \dthru)$ has no maximal subquandles.
\end{proposition}
\begin{proof}
Fix arbitrary $M\psq \Q$. We use the fact that maximal subquandles must be complemented by every singleton in their set-theoretic complement and show that this cannot hold for $M$. We will also use an analagous fact for groups:
\begin{lemma} \label{maximalsubgroup}
Let $G$ be a group. Then, $M\lneq G$ is maximal iff $\gen{M\cup\{x\}}=G$ for each $x\in G\setminus M$.
\end{lemma}
\begin{proof}
The forward direction is clear. Now suppose $\gen{M\cup\{x\}}=G$ for each $x\in G\setminus M$. Fix arbitrary $N\leq G$ properly extending $M$. Then, there is some nonempty $S\subseteq G\setminus M$ such that $N=\gen{M\cup S}.$ Fixing some $s\in S$, we see
$$G=\gen{M\cup\{s\}}\subseteq \gen{M\cup S} = N,$$
so $N=G$. Hence, the only subgroup of $G$ properly extending $M$ is $G$ itself, so $M$ is maximal.
\end{proof}
By Theorem \ref{QRacknotComp}, $\{0\}$ has no complement in $\cL(\Q)$, so $\{0\}$ is not a complement of $M$. Thus, if $0\notin M$, $M$ cannot be maximal. Now suppose $0 \in M.$ By Lemma 5.1 in \cite{SakiKiani2021}, $M$ is a subgroup of $(\Q,+).$ Since $\Q$ has no maximal subgroups, by Lemma \ref{maximalsubgroup}, there is some $q\in \Q\setminus M$ such that $\gen{M\cup\{q\}}$ is a proper subgroup of $\Q.$ Hence, $\qgen{M \cup\{q\}}\subseteq \gen{M\cup\{q\}}\subset Q$, so $\{q\}$ is not a complement of $M$ in $\cL(Q)$. Therefore, $M$ cannot be maximal.
\end{proof}

\subsection{Ind-finite Quandles}
Although $\Q$ is an infinite quandle with a non-complemented sublattice, $\Q$ does not have any clear relationship to finite quandles, which do have complemented sublattices. In this section, we capture one type of relationship an infinite quandle can have with finite quandles by examining ind-finite quandles. We also provide an example which shows that the subquandle lattices of ind-finite quandles are not necessarily complemented.
\begin{definition}
A set $(A,\leq)$ is \textit{directed} if for all $a,b\in A$, there exists some $c\in A$ such that $a\leq c$ and $b\leq c$. In other words, every pair of elements in $A$ has an upper bound.
\end{definition}
\begin{definition}
A nonempty quandle $Q$ is \textit{ind-finite} if there is a family of finite quandles $\{Q_i\}_{i \in \cI}$ indexed by a directed set $(\cI, \leq)$ such that for all $i,j\in\cI$,
\begin{enumerate}
    \item $Q_i\sq Q$,
    \item  $|Q_i|<\infty$,
    \item $i\leq j$ implies $Q_i\sq Q_j$, and
    \item $Q=\bigcup_{i\in\cI} Q_i.$ 
\end{enumerate}
\end{definition}
There is an equivalent formulation of ind-finiteness which will soon be useful. This formulation requires the Axiom of Dependent Choice.
\begin{lemma}
Suppose $Q$ is a nonempty quandle. Then, $Q$ is ind-finite if and only if every finitely generated subquandle of $Q$ is finite.
\end{lemma}
\begin{proof}
Suppose that $Q$ is ind-finite with $Q=\bigcup_{i\in\cI} Q_i$,
and fix an arbitrary finitely generated subquandle $Q'$ of $Q$. Let $\{g_1,\dots,g_n\}$ be a generating set for $Q'$. Then, for each $1\leq j \leq n$, since $g_j\in Q$, there exists some $i(j)\in \cI$ such that $g_j \in Q_{i(j)}$. Clearly, 
$$Q'=\qgen{\{g_1,g_2,\dots, g_n\}} \subseteq \qgen{Q_{i(1)},\dots,Q_{i(n)}}=\qgen{\cup_{1\leq j\leq n} Q_{i(j)}}.
$$ But, since $\sq$ is a total order on $\{Q_i\}_{i\in\cI}$, $\bigcup_{1\leq j\leq n} Q_{i(j)} = Q_m$,
where $m=\max_{1\leq j\leq n} i(j).$ Thus, $\qgen{Q_{i(1)},\dots,Q_{i(n)}} = \qgen{Q_m}=Q_m$. Since $Q_m$ is finite by assumption and $Q'\subseteq Q_m$, $Q'$ must also be finite.
\\\\
Now suppose that every finitely generated subquandle of $Q$ is finite. Let $x\in Q$ be arbitrary, and define $Q_0 = \qgen{x}$ and $\cI = |Q|$ (which is a cardinal that isn't necessarily finite). For each $i \in \cI$, while $\bigcup_{k\leq i} Q_k \neq Q,$ let $Q_{i+1}=\qgen{Q_i, y}$ for some $y\in Q\setminus \bigcup_{k\leq i} Q_k.$ Since each $Q_i$ is finitely generated subquandle of $Q$ by construction, each $Q_i$ must also be finite. Moreover, $\bigcup_{i\in |Q|} Q_i = Q$, since every element of $Q$ can be reached by this construction, and $Q_i\psq Q_{i+1}$ for each $i\in\cI$. By definition, $Q$ is ind-finite.
\end{proof}
As a corollary, we find that the infinite quandle $\Q$ cannot be ind-finite. 
\begin{corollary}
$(\Q, \dthru)$ is not ind-finite.
\end{corollary}
\begin{proof}
Fix any $x,y\in \Q$ such that $x<y$. By part (5) of Theorem 3.3 in \cite{SakiKiani1}, 
$$\qgen{x,y}=\{r_0x+r_1y\},$$
where $r_0,r_1\in \Z$, exactly one of $r_0,r_1$ is odd, and $r_0+r_1=1$. In particular, $\{y+j(y-x) : j \in \N\} \subseteq \qgen{x,y}.$ Hence, $\qgen{x,y}$ is finitely-generated and infinite subquandle of $\Q$, so by the previous lemma, $\Q$ cannot be ind-finite.
\end{proof}
But, it turns out that we can transform $\Q$ quite easily to create an ind-finite quandle. We will also see that this transformed quandle maintains the properties of having a non-complemented subquandle lattice and no maximal subquandles, thus proving that the subquandle lattices of ind-finite quandles are not necessarily complemented.
\begin{proposition}
$(\Q / \Z, \dthru)$ is ind-finite.
\end{proposition}
\begin{proof}
It is easy to see that when an abelian group is made into a quandle using the Takasaki operation, any subgroup is a subquandle. Therefore, $\langle\frac{1}{n!}+\Z\rangle$ is a (finite) subquandle of $\Q/\Z$ for each $n\in\mathbb{N}.$ Observe then that $\Q/\Z$ is the ascending union of the subgroups generated by the cosets of the reciprocals of factorials, i.e., $\Q/\Z=\bigcup_{n\in\mathbb{N}}\langle\frac{1}{n!}+\Z\rangle$,  and thus is easily seen to be ind-finite.
\end{proof}
\begin{proposition}
$(\Q / \Z, \dthru)$ does not have a complemented subquandle lattice and has no maximal subquandles.
\end{proposition}
\begin{proof}
Note that $\Q/\Z=\{[\frac{m}{n}]: \frac{m}{n}\in \Q\},$ where $[\frac{m}{n}]=\frac{m}{n}+\Z.$ We are assuming $\Q/\Z$ is equipped with the (well-defined) operation $[x]\dthru [y]=[2x-y]$, where $x,y\in \Q$. Thus, in the proofs of Lemma 5.1 and Theorem 5.2 in \cite{SakiKiani2021}, we may replace the elements $x\in \Q$ with their corresponding cosets $[x]\in \Q/\Z$ to obtain the following analogous results for $\Q/\Z$:
\begin{enumerate}
    \item The nonempty subquandles of $\Q/\Z$ are exactly the cosets of the subgroups of $\Q/\Z.$ 
    \item $\cL((\Q/\Z,\dthru))$ is not complemented. In particular, $\{\Z\}$ has no complement.
\end{enumerate}
To show $\Q/\Z$ has no maximal subquandles, fix an arbitrary $M\psq \Q/\Z$ --- we show that $M$ cannot be maximal. As in  Proposition \ref{QNoMaximal}, if $\Z\notin M$, $M$ cannot be maximal. Additionally, if $\Z\in M$, note that $(\Q/\Z,+)$ has no maximal subgroups, so the same argument from Proposition \ref{QNoMaximal} shows that $M$ cannot be maximal.
\end{proof}
\subsection{Profinite Quandles}
Profinite quandles are the dual to ind-finite quandles --- just as ind-finite quandles are the direct limit of a directed system of finite quandles, profinite quandles are the inverse limit of an inverse system of finite quandles. In this section, we formalize the notion of a profinite quandle and provide several examples of infinite profinite quandles. We leave the complementedness of profinite quandle lattices as a question for future research.

\begin{definition}
Let $(\cI, \leq)$ be a directed set, and let $\{Q_i\}_{i \in \cI}$ be a family of quandles indexed by $\cI$. Suppose we have a family of quandle homomorphisms $\{f_{ij}:Q_j\to Q_i:i,j\in \cI, \ i\leq j\}$ such that for all $i,j,k\in\cI$,
\begin{enumerate}
    \item $f_{ii} = \mathrm{id}_{Q_i}$, and
    \item If $i \leq j \leq k$, $f_{ik} = f_{ij}\circ f_{jk}$.
\end{enumerate}

    \adjustbox{center}{
    \begin{tikzcd}[row sep=large,column sep=huge]
    Q_k \arrow[rd,"f_{ik}"] \arrow[r, "f_{jk}"] & Q_j \arrow[d,"f_{ij}"] \\
    & Q_i
    \end{tikzcd}}

Then, the pair $(\{Q_i\}_{i \in \cI}, \{f_{ij}\}_{i\leq j \in \cI})$ is called an \textit{inverse system} of quandles and morphisms over $\cI.$
\end{definition}
\begin{definition} \label{limit}
Given an inverse system $(\{Q_i\}_{i \in \cI}, \{f_{ij}\}_{i\leq j \in \cI})$, we define the \textit{inverse limit} of the system as follows:
$$Q=\lim _{\overleftarrow{i \in \cI}} Q_{i}=\left\{\vec{q} \in \prod_{i \in \cI} Q_{i} : q_{i}=f_{i j}\left(q_{j}\right) \text { for all } i \leq j \text { in } \cI\right\}.$$
\end{definition}
We say a quandle $Q$ is \textit{profinite} it is the inverse limit of the inverse system composed of a family of finite quandles and their morphisms. Indeed, this construction always yields a quandle:
\begin{proposition}
Profinite quandles are quandles under coordinatewise operations.
\end{proposition}
\begin{proof}
Suppose that $Q=\displaystyle\lim _{\overleftarrow{i \in \cI}} Q_{i},$ where $(\{Q_i\}_{i \in \cI}, \{f_{ij}\}_{i\leq j \in \cI})$
is an inverse system of finite quandles. Also suppose the associated quandle operation for each $Q_i$ is $\thru_i.$ Because each $Q_i$ is a quandle under $\thru_i$, axioms Q1, Q2, and Q3 hold for $Q$ since they hold in each of its coordinates. But, since $Q$ may be infinite, we must also verify that $Q$ is closed under its operation $\thru$. Without loss of generality, let $\cI=\N$. Also let $(q_1,q_2,\dots),(r_1,r_2,\dots)\in Q$ be arbitrary. Then,
\begin{align*}
    (q_1,q_2,\dots) \thru (r_1,r_2,\dots) &= (q_1\thru_1 r_1,q_2\thru_2 r_2,\dots),
\end{align*}
so, by the definition of $Q$, it suffices to show that for arbitrary $i,j\in \cI$, $f_{ij}(q_j \thru_j r_j)=q_i\thru_i r_i$. But, because $f_{ij}$ is a quandle homomorphism, $f_{ij}(q_j \thru_j r_j)=f_{ij}(q_j)\thru_i f_{ij}(r_j)=q_i\thru_i r_i$. Hence, $Q$ is closed under $\thru$, as desired.
\end{proof}
Next, we provide a canonical way to construct a profinite quandle from a profinite abelian group: consider the group as a Takasaki kei, as defined in Example \ref{dihedralKei}.

\begin{proposition}
\label{Profinite Construction}
Suppose $(\{Q_i\}_{i \in \cI}, \{f_{ij}\}_{i\leq j \in \cI})$ is an inverse system of finite abelian groups and their morphisms. Then, its inverse limit $Q$ is a profinite Takasaki kei under coordinatewise operations.
\end{proposition}
\begin{proof}
Since each $Q_i$ is an abelian group, it is a Takasaki kei. Additionally, each group homomorphism $f_{ij}$ is a quandle homomorphism by following lemma:
\begin{lemma}
Suppose $(Q_1,\dthru)$ and $(Q_2,\dthru)$ are Takasaki quandles, and suppose $(Q_1,+)$ and $(Q_2,+)$ are the corresponding groups. Then, any group homomorphism $f:Q_1\to Q_2$ is also a quandle homomorphism.
\end{lemma}

This lemma proves that regarding an abelian group as a Takasaki quandle is a functor into $\mathbf{Kei}$, the full subcategory of $\mathbf{Quand}$ whose objects are keis. We shall give it a name.

\begin{definition} The
\textit{Takasaki functor}, denoted $\Tak:\mathbf{Ab}\rightarrow \mathbf{Kei}$, assigns to each abelian group the Takesaki kei.
\end{definition}

Thus, $(\{Q_i\}_{i \in \cI}, \{f_{ij}\}_{i\leq j \in \cI})$ is also an inverse system of finite Takasaki quandles. To show that the inverse limit of this inverse system is a profinite quandle, it suffices to show that the functor $\Tak$ has a left adjoint, as this will show it preserves all limits. We shall construct this left adjoint explicitly.

\begin{definition}
Let $K$ be a kei. Then we define $\AdTak(K)$ to be the quotient group
\[\AdTak(K) = \langle K\rangle_{Ab}/\langle x\thru y - 2y + x: x,y\in K\rangle_{Ab}\]
\end{definition}


It is easy to see that $\AdTak$ is a left adjoint to $\Tak$. 

\begin{proposition}
Let $\eta$ be the inclusion $x\mapsto \overline{x}$. Also let $A$ be an abelian group, and let $\phi$ be a morphism of keis $ \phi:K\to \Tak(A) $. Then $\phi$ factors uniquely through $\Tak(\phi')$ for a unique morphism of abelian groups $\phi'$, as in the commutative diagram below. 
\end{proposition}
\[\begin{tikzcd}
	K & {\Tak(\AdTak(K))} \\
	& {\Tak(A)}
	\arrow["\phi"', from=1-1, to=2-2]
	\arrow["\eta", from=1-1, to=1-2]
	\arrow["{\Tak(\phi')}", dotted, from=1-2, to=2-2]
\end{tikzcd}\]
\begin{proof}
    The map $\eta$ can be thought of as the composition (in sets) of the inclusion $K\hookrightarrow \AdTak(K)$ along with $p:\langle K \rangle_{Ab} \to \langle K \rangle_{Ab}/\langle x\thru y - 2y + x: x,y\in K\rangle_{Ab}$ (regarded as a function between underlying sets). By the universal property of free abelian groups, the function $\phi$ induces a unique group homomorphism $ \phi': \langle K \rangle_{Ab} \to A$. By definition of a Takasaki kei, for any elements $x,y\in K$, we have $\phi(x\thru y) = \phi(x)\thru \phi(y) = 2\phi(y) - \phi(x)$. Hence $ \phi(x\thru y) - 2\phi(y) + \phi(x) = \phi(x\thru y - 2y + x) = 0 $. Therefore, 
    $\langle x\thru y - 2y + x: x,y\in K\rangle_{Ab}\subseteq \ker(\phi')$, which implies $\langle x\thru y - 2y + x: x,y\in K\rangle_{Ab}\normal \ker(\phi')$. By the universal property of quotients in abelian groups, $\phi'$ induces yet another (unique) homomorphism $\mathrm{Tak}(\phi'):\AdTak(K) \to A $ of abelian groups.
\end{proof}

 Therefore, by Proposition \ref{Profinite Construction}, the inverse limit of $(\{Q_i\}_{i \in \cI}, \{f_{ij}\}_{i\leq j \in \cI})$ is a profinite quandle under coordinatewise operations.
\end{proof}

Because $\Tak$ admits a left adjoint, it preserves all limits, in particular the limits giving a profinite abelian group as the limit of the diagram of all its finite quotient groups.  This observation gives rise to a large class of profinite keis:

\begin{example} \leavevmode
\begin{enumerate} 
    
    \item Suppose $p\in \N$ is a prime. Then, the $p$-adic integers $\Z_p$ form a profinite Takasaki quandle.
    \item The additive group of the ring of profinite integers $\hat{\Z}=\prod_{p} \Z_p$ forms a profinite Takasaki quandle.
\end{enumerate}
\end{example}

The same argument applied to the functor $\Conj$ gives:
\begin{example}
If $G$ is a profinite group, then $\Conj(G)$ is a profinite quandle.  
\end{example}


Clearly some pro-finite quandles have complemented subquandle lattices. For example, besides finite quandles \cite{SakiKiani2021}, trivial quandles with an infinite profinite set (that is an infinite totally disconnected compact Hausdorff space) as an underlying set are profinite quandles, being the inverse limit of a diagram of finite trivial quandles. Moreover, these quandles, as do all trivial quandles, have the (necessarily complemented) Boolean algebra of all subsets equipped with trivial quandle structures as a subquandle lattice.

However, in general, if $Q$ is infinite, it is hard to determine if $\cL(Q)$ is complemented, as even the lattice of subquandles of $\cL(Q)$ is difficult to examine because the underlying sets of either class of examples above are uncountable.


 Nonetheless, we conjecture that the subquandle lattice of a profinite quandle is complemented.

\section{Acknowledgments}
The authors wish to thank the National Science Foudation for financial support for the conduct of this research under NSF Grant DMS-1659123.  The first four authors wish to thank the fifth author, along with Marianne Korten and Kim Klinger-Logan, for their mentorship and support during and after the SUMaR 2023, and Kansas State University for its hospitality during the REU.

\begin{section}{Bibliography}

\end{section}

\end{document}